\documentclass[]{siamart171218}       
\usepackage{todonotes}
\usepackage{algorithmic}
\usepackage{algorithm}
\usepackage{amsmath}
\usepackage{amssymb}
\usepackage{graphicx}
\usepackage{listings}
\usepackage{wrapfig}
\usepackage{tikz}

\newcommand{\X}{{\bf x}}
\newcommand{\Y}{{\bf y}}
\newcommand{\Z}{{\bf z}}
\newcommand{\V}{{\bf v}}
\newcommand{\W}{{\bf w}}
\newcommand{\A}{{\bf a}}
\newcommand{\B}{{\bf b}}
\renewcommand{\S}{{\bf s}}
\newcommand{\T}{{\bf t}}

\newcommand{\R}{{ {\rm I} \kern -.25em {\rm R} }}

\newcommand{\ass}{=}

\usepackage{graphicx}
\begin{document}

\title{A Note on Adjoint Linear Algebra}

\author{Uwe Naumann \thanks{Department of Computer Science, RWTH Aachen University, 52056 Aachen, Germany, \email{naumann@stce.rwth-aachen.de}  }}

\maketitle

\begin{abstract}
A new proof for adjoint systems of linear equations is presented. The argument
is built on the principles of Algorithmic Differentiation. 
Application to scalar multiplication sets the base line.
Generalization yields adjoint inner vector, matrix-vector, and matrix-matrix 
products leading to an alternative proof for first- as well as higher-order 
adjoint linear systems. 
\end{abstract}


\section{Motivation}
\label{sec:motivation}

Algorithmic Differentiation \cite{Griewank2008EDP,Naumann2012TAo}
of numerical programs builds on a set of
elemental functions with known partial derivatives with respect to their 
arguments at the given point of evaluation.
The propagation of adjoint derivatives relies on the associativity of the 
chain rule of differential calculus. Differentiable combinations of 
elemental functions yield higher-level elementals. Efficient implementation 
of AD requires the highest possible level of elemental functions.

Basic AD assumes the set of elemental functions to be formed by the arithmetic
operators and intrinsic functions built into the given programming language.
While its application to linear algebra methods turns out to be straight 
forward basic AD is certainly not the method of choice from the point of
view of computational efficiency. Elementals of the highest possible level
should be used. Their derivatives should be formulated as functions of
high-level elementals in order to exploit benefits of corresponding 
optimized implementations.

Following this rationale this note presents a new way to derive 
adjoint systems of linear equations based on adjoint Basic Linear Algebra Subprograms (BLAS) \cite{Lawson1979BLA}. 
It is well known (see \cite{Giles2008CMD} and references therein) that 
for systems $A \cdot \X = \B$ of $n$ linear equations 
with invertible $A$ and {\em primal} solution
	$\X \ass A^{-1} \cdot \B$
first-order {\em adjoints} $A_{(1)}$ of $A$ 
(both $\in \R^{n \times n}$ with $\R$ denoting the real numbers)
and $\B_{(1)}$ of $\B$ (both $\in \R^n$) 
can be evaluated at the primal solution $\X \in \R^n$ as 
\begin{equation} \label{als}
	\left (
	\begin{split}
		\B_{(1)} &\ass A^{-T} \cdot \X_{(1)} \\
		A_{(1)} &\ass -\B_{(1)} \cdot \X^T 
	\end{split}
	\right ) \; .
\end{equation}
The main contribution of this note is an alternative
proof for Eqn.~(\ref{als}) that builds naturally on the adjoint 
BLAS used 
in the context of state of the art AD. 
For consistency with related work we follow the notation 
in \cite{Naumann2012TAo}, that is, $v^{(1)} \in \R$ denotes the value of the
first-order directional derivative (or tangent) associated with a 
variable $v\in \R$ and $v_{(1)} \in \R$ denotes the value of its adjoint.

\section{Prerequisites}

The Jacobian $\nabla F = \nabla F(\X) \equiv \frac{d F}{d \X}(\X) \in \R^{m \times n}$ of a differentiable implementation of
$\Y \ass F(\X) : \R^n \rightarrow \R^m$ as a computer program induces a 
linear mapping
$\Y^{(1)} \ass \nabla F \cdot \X^{(1)}: \R^n \rightarrow \R^m$ implementing the
tangent of $F.$ 
The corresponding {\em adjoint} operator $\nabla F^*=\nabla F^*(\X)$ is formally defined via
the inner vector product identity 
\begin{equation} \label{eqn:adj}
	\langle \nabla F \cdot \X^{(1)},\Y_{(1)} \rangle 
	=
	\langle \X^{(1)}, \nabla F^* \cdot \Y_{(1)} \rangle 
\end{equation}
yielding
$\nabla F^* =\nabla F^T$ \cite{MR0117523}.
In the following all (program) variables are assumed to be 
alias- and context-free, that is, distinct variables do not overlap
in memory and $F$ is assumed to be not embedded in an enclosing 
computation. We distinguish between {\em active} and {\em passive} variables. 
Derivatives of all active outputs of the 
given program are computed with respect to all active inputs. We are not 
interested in derivatives of passive outputs nor are we computing derivatives 
with respect to passive inputs.

\section{BLAS Revisited}

In its basic form AD builds on known tangents and adjoints of the arithmetic
functions and operators built into programming languages. Tangents and 
adjoints
are propagated along the flow of data according to the chain rule of 
differential calculus.
We enumerate entries of vectors $\V \in \R^n$ staring from zero
as $v_0,\ldots,v_{n-1}.$

From the perspective of AD adjoint versions of higher-level BLAS are derived as
adjoints of lower-level BLAS. Optimization of the result aims for 
implementation using the highest possible level of BLAS. For example, 
adjoint matrix-matrix multiplication (level-3 BLAS) is derived from adjoint 
matrix-vector multiplication (level-2 BLAS) yielding efficient evaluation as 
two matrix-matrix products (level-3 BLAS) as shown in Lemma~\ref{lem:AX}.
Rigorous derivation of this result requires bottom-up investigation of the BLAS hierarchy. We start with basic scalar multiplication (Lemma~\ref{lem:ax}) 
followed by the inner vector (Lemma~\ref{lem:aTx}) and matrix-vector (Lemma~\ref{lem:Ax}) products as prerequisites for the matrix-matrix product.

\begin{lemma} \label{lem:ax}
The adjoint of scalar multiplication
$y \ass a \cdot x$ with active $a,x,y \in \R$ is computed as
\begin{equation} \label{eqn:ax}
\begin{split}
a_{(1)} &\ass x \cdot y_{(1)} \\
x_{(1)} &\ass a \cdot y_{(1)}
\end{split} 
\end{equation}
for $y_{(1)} \in \R$  yielding $a_{(1)},x_{(1)} \in \R.$ 
\end{lemma}
\begin{proof}
Differentiation of $y=a \cdot x$ with respect to
$a$ and $x$ yields the tangent
$$
y^{(1)}=
\left \langle 
\begin{pmatrix}
a^{(1)} \\
x^{(1)} 
\end{pmatrix} ,
\begin{pmatrix} x \\ a \end{pmatrix}  \right \rangle
$$
for 
$y^{(1)},a^{(1)},x^{(1)} \in \R.$ 
Eqn.~(\ref{eqn:adj}) implies
\begin{align*}
\langle y^{(1)}, y_{(1)} \rangle &= 
y^{(1)} \cdot y_{(1)}= \left \langle 
\begin{pmatrix}
a^{(1)} \\
x^{(1)} 
\end{pmatrix},
\begin{pmatrix}
a_{(1)} \\
x_{(1)} 
\end{pmatrix}
\right \rangle =
\left \langle 
\begin{pmatrix}
a^{(1)} \\
x^{(1)} 
\end{pmatrix} ,
\begin{pmatrix} x \\ a \end{pmatrix}  \right \rangle
\cdot y_{(1)} 
\end{align*}
yielding $$ \begin{pmatrix}a_{(1)} \\x_{(1)} \end{pmatrix} = \begin{pmatrix} x \\ a \end{pmatrix} \cdot y_{(1)} $$ 
and hence Eqn.~(\ref{eqn:ax}). 
\end{proof}

\begin{lemma} \label{lem:aTx}
The adjoint of 
an inner vector product 
$$y\ass \langle \A, \X \rangle \equiv \A^T \cdot \X = \sum_{i=0}^{n-1} a_i \cdot x_i$$ 
with active inputs $\A \in \R^n$ and $\X \in \R^n$ 
yielding the active output $y \in \R$ is computed as
\begin{equation} \label{eqn:aTx}
\begin{split}
\A_{(1)} &\ass \X \cdot y_{(1)} \\
\X_{(1)} &\ass \A \cdot y_{(1)}
\end{split} 
\end{equation} 
for $y_{(1)} \in \R$ yielding
$\A_{(1)} \in \R^n $ and $\X_{(1)} \in \R^n.$ 
\end{lemma} 
\begin{proof}
Differentiation of $y=\A^T \cdot \X,$ for
$\A=(a_i)_{i=0,\ldots,n-1}$ and $\X=(x_i)_{i=0,\ldots,n-1}$,
with respect to
$\A$ and $\X$ yields the tangent 
\begin{align*}
y^{(1)}&=\sum_{i=0}^{n-1} (x_i~a_i) \cdot 
\begin{pmatrix}
a^{(1)}_i \\
x^{(1)}_i 
\end{pmatrix} =
\sum_{i=0}^{n-1} \left ( x_i \cdot a^{(1)}_i + a_i \cdot x^{(1)}_i \right )\\ &=
\sum_{i=0}^{n-1} x_i \cdot a^{(1)}_i + \sum_{i=0}^{n-1} x^{(1)}_i \cdot a_i 
=
\X^T \cdot \A^{(1)} + \A^T \cdot \X^{(1)} =
(\X^T~\A^T) \cdot \begin{pmatrix}
\A^{(1)} \\
\X^{(1)} 
\end{pmatrix}\; .
\end{align*}
Eqn.~(\ref{eqn:adj}) implies
$$
y_{(1)} \cdot y^{(1)}=
(\A_{(1)}^T~\X_{(1)}^T) \cdot 
\begin{pmatrix}
\A^{(1)} \\
\X^{(1)} 
\end{pmatrix}
=
y_{(1)} \cdot (\X^T~\A^T) \cdot 
\begin{pmatrix}
\A^{(1)} \\
\X^{(1)} 
\end{pmatrix}
$$
yielding $(\A_{(1)}^T~\X_{(1)}^T)=y_{(1)} \cdot (\X^T~\A^T)$
and hence Eqn.~(\ref{eqn:aTx}). 
\end{proof}

The following derivation of adjoint matrix-vector and matrix-matrix
products relies on serialization of matrices.
Individual rows of a matrix $A \in \R^{m \times n}$ are denoted as 
$\A_i \in \R^{1 \times n}$ for $i=0,\ldots,m-1;$
columns are denoted as  $\A^j \in \R^m$ for $i=0,\ldots,n-1.$
(Row) Vectors in $\R^{1 \times n}$ are denoted as 
$\left (v_j \right )^{j=0,\ldots,n-1};$
(column) vectors in $\R^m$ are denoted as 
$\left (v_i \right )_{i=0,\ldots,m-1};$
Consequently, a row-major serialization of $A$ is given 
by $\left (\A^T_i \right )_{i=0,\ldots,m-1}.$
A column-major serialization of $A$ is given by 
$\left (\A^j \right )_{j=0,\ldots,n-1}.$ 
Tangents and adjoints of the individual entries of $A$ define
$$A^{(1)}=(\A^{(1)}_i)_{i=0,\ldots,m-1}=(a^{(1)}_{i,j})_{i=0,\ldots,m-1}^{j=0,\ldots,n-1}$$ and
$$A_{(1)}=(\A_{(1)i})_{i=0,\ldots,m-1}=(a_{(1)i,j})_{i=0,\ldots,m-1}^{j=0,\ldots,n-1},$$ 
respectively. 

\begin{lemma} \label{lem:Ax}
The adjoint of 
a matrix-vector product 
$$\Y\ass A \cdot \X \equiv \left ( \A_i \cdot \X \right )_{i=0,\ldots,m-1}$$ 
with active inputs $A\in \R^{m \times n}$ 
and $\X \in \R^n$ yielding the active output $\Y\in \R^m$ is computed as
\begin{equation} \label{eqn:Ax}
\begin{split}
\X_{(1)} &\ass A^T \cdot \Y_{(1)} \\
A_{(1)} &\ass \Y_{(1)} \cdot \X^T
\end{split} 
\end{equation}
for $\Y_{(1)}\in \R^m$ yielding
$\X_{(1)} \in \R^n$ and 
$A_{(1)}\in \R^{m \times n}.$ 
\end{lemma}
\begin{proof}
Differentiation of $\Y=A \cdot \X,$ 
where $A=\left (\A_i \right )_{i=0,\ldots,m-1},$ 
$\X=\left (x_j \right )_{j=0,\ldots,n-1}$ and $\Y=\left (y_i \right )_{i=0,\ldots,m-1},$ with respect to
$A$ and $\X$ yields the tangent 
\begin{align*}
\Y^{(1)}&=\left (  
\left \langle
\begin{pmatrix}
\X \\
\A^T_i
\end{pmatrix}
,\begin{pmatrix}
{\A_i^{(1)}}^T \\
\X^{(1)} 
\end{pmatrix}
\right \rangle
\right )_{i=0,\ldots,m-1} 
=
\left (  
\X^T \cdot {\A_i^{(1)}}^T+\A_i \cdot \X^{(1)}
\right )_{i=0,\ldots,m-1} \\
&=
\left (  
\X^T \cdot {\A_i^{(1)}}^T 
\right )_{i=0,\ldots,m-1} 
+
\left (  
\A_i \cdot \X^{(1)}
\right )_{i=0,\ldots,m-1} \\
&=
\left (  
\A_i^{(1)} \cdot \X 
\right )_{i=0,\ldots,m-1} 
+
\left (  
\A_i \cdot \X^{(1)}
\right )_{i=0,\ldots,m-1} \\
&=
\left (  
\A_i^{(1)} 
\right )_{i=0,\ldots,m-1} \cdot \X 
+
\left (  
\A_i 
\right )_{i=0,\ldots,m-1} \cdot \X^{(1)} 
=A^{(1)} \cdot \X + A \cdot \X^{(1)} \; .
\end{align*}
Eqn.~(\ref{eqn:adj}) implies
\begin{align*}
\left \langle \Y_{(1)}, \Y^{(1)} \right \rangle &=
\left \langle
\begin{pmatrix}
\left (\A^T_{(1)i} \right )_{i=0,\ldots,m-1} \\
\X_{(1)}
\end{pmatrix}, 
\begin{pmatrix}
\left (\A^{(1)T}_i \right )_{i=0,\ldots,m-1} \\
\X^{(1)} 
\end{pmatrix}
\right \rangle \\
&=
\left (\A^T_{(1)i} \right )^T_{i=0,\ldots,m-1} \cdot \left (\A^{(1)T}_i \right )_{i=0,\ldots,m-1} + \X_{(1)}^T \cdot \X^{(1)} \\
&=\Y_{(1)}^T \cdot \left (A^{(1)} \cdot \X + A \cdot \X^{(1)} \right ) 
=\Y_{(1)}^T \cdot A^{(1)} \cdot \X + \Y_{(1)}^T \cdot A \cdot \X^{(1)} \\
&=\left ( \left ( y_{(1)i} \cdot \A^{(1)T}_i \right )_{i=0,\ldots,m-1} \right )^T
\cdot \left (\X \right )_{i=0,\ldots,m-1} + \Y_{(1)}^T \cdot A \cdot \X^{(1)} \\
&=\underset{=\left ( (\A^T_{(1)i})_{i=0,\ldots,n-1}\right )^T}{\underbrace{\left ( \left ( y_{(1)i} \cdot \X \right )_{i=0,\ldots,m-1} \right )^T}}
\cdot \left (\A^{(1)T}_i \right )_{i=0,\ldots,m-1} + \underset{=\X^T_{(1)}}{\underbrace{\Y_{(1)}^T \cdot A}} \cdot \X^{(1)} \; ,
\end{align*}
where $\left (\X \right )_{i=0,\ldots,m-1} \in \R^{m \cdot n}$ denotes a 
concatenation of $m$ copies of $\X \in \R^n$ as a column vector.
Eqn.~(\ref{eqn:Ax}) follows immediately.
\end{proof}
\begin{lemma} \label{lem:AX}
The adjoint of a matrix-matrix product $Y \ass A \cdot X$ 
with active inputs
$A \in \R^{m \times p},$ 
	$X \in \R^{p \times n}$ yielding the active output 
$Y \in \R^{m \times n}$ is computed as
\begin{equation} \label{eqn:AX}
\begin{split}
A_{(1)} &\ass Y_{(1)} \cdot X^T \\
X_{(1)} &\ass A^T \cdot Y_{(1)}  
\end{split}
\end{equation}
for 
$Y_{(1)} \in \R^{m \times n}$ yielding 
$A_{(1)} \in \R^{m \times p}$ and $X_{(1)} \in \R^{p \times n}.$ 
\end{lemma}
\begin{proof}
Differentiation of $Y=A \cdot X,$ where
$A=\left (\A_i \right )_{i=0,\ldots,m-1}$, $X=\left (\X^k \right )^{k=0,\ldots,p-1}$ and $Y=\left (\Y^k \right )^{k=0,\ldots,p-1},$
with respect to
$A$ and $X$ yields tangents
\begin{align*}
\Y^{(1)k}&=\left (  
\left \langle
\begin{pmatrix}
\X^k \\
\A^T_i
\end{pmatrix}
,\begin{pmatrix}
{\A_i^{(1)}}^T \\
\X^{(1)k} 
\end{pmatrix}
\right \rangle
\right )_{i=0,\ldots,m-1} 
=
A^{(1)} \cdot \X^k + A \cdot \X^{(1)k} \; .
\end{align*}
for $k=0,\ldots,p-1$ and hence
$$
Y^{(1)}=A^{(1)} \cdot X + A \cdot X^{(1)} \; .
$$
Eqn.~(\ref{eqn:adj}) implies
\begin{align*}
\left \langle \Y_{(1)}^k, \Y^{(1)k} \right \rangle &=
\left \langle
\begin{pmatrix}
\left (\A^T_{(1)i} \right )_{i=0,\ldots,m-1} \\
\X_{(1)k}
\end{pmatrix}, 
\begin{pmatrix}
\left (\A^{(1)T}_i \right )_{i=0,\ldots,m-1} \\
\X^{(1)k} 
\end{pmatrix}
\right \rangle \\
&=\underset{=\left ( (\A^T_{(1)i})_{i=0,\ldots,m-1}\right )^T}{\underbrace{\left ( \left ( y_{(1)i}^k \cdot \X^k \right )_{i=0,\ldots,m-1} \right )^T}}
	\cdot \left (\A^{(1)T}_i \right )_{i=0,\ldots,m-1} + \underset{=\X^{k^T}_{(1)}}{\underbrace{\Y_{(1)}^{k^T} \cdot A}} \cdot \X^{(1) k}
\end{align*}
for $k=0,\ldots,p-1$ and hence the Eqn.~(\ref{eqn:AX}). 
\end{proof}

\section{Systems of Linear Equations Revisited}

Lemmas~\ref{lem:ls1} and \ref{lem:ls2} form the basis for the new proof of
Eqn.~(\ref{als}).

\begin{lemma} \label{lem:ls1}
The tangent $$Y^{(1)} = A \cdot X^{(1)} \cdot B$$
of $Y \ass A \cdot X \cdot B$ 
for active $X \in \R^{n \times q},$ $Y \in \R^{m \times p}$ and
passive $A \in \R^{m \times n},$ $B \in \R^{q \times p}$ implies
the adjoint 
$$X_{(1)} = A^T \cdot Y_{(1)} \cdot B^T \; .$$ 
\end{lemma} 
\begin{proof}
$$
Y^{(1)} = Z^{(1)} \cdot B \quad \Rightarrow \quad 
Z_{(1)} = Y_{(1)} \cdot B^T 
$$ 
follows from application of Lemma~\ref{lem:AX} to $Y=Z \cdot B$ with
passive $B.$
$$
Z^{(1)} = A \cdot X^{(1)} \quad \Rightarrow \quad 
X_{(1)} = A^T \cdot Z_{(1)} 
$$
follows from application of Lemma~\ref{lem:AX} to $Z=A \cdot X$ with
passive $A.$ Substitution of $Z^{(1)}$ and $Z_{(1)}$ yields 
Lemma~\ref{lem:ls1}. 
\end{proof}

\begin{lemma} \label{lem:ls2}
The tangent
$$
	Y^{(1)} = \sum_{i=0}^{k-1} A_i \cdot X_i^{(1)} \cdot B_i 
$$ 
of $
Y = \sum_{i=0}^{k-1} A_i \cdot X_i \cdot B_i = \sum_{i=0}^{k-1} Y_i 
$
with active
$X_i \in \R^{n_i \times q_i},$ $Y \in \R^{m \times p}$
and with passive
$A_i \in \R^{m \times n_i},$ $B_i \in \R^{q_i \times p}$
implies the adjoint
$$
	X_{i(1)} = A_i^T \cdot Y_{(1)} \cdot B_i^T
$$ for $i=0,\ldots,k-1$
\end{lemma}
\begin{proof}
From $$Y_i^{(1)} = A_i \cdot X_i^{(1)} \cdot B_i$$
follows with 
Lemma~\ref{lem:ls1}
$$X_{i(1)} = A^T_i \cdot Y_{i(1)} \cdot B^T_i$$
for $i=0,\ldots,k-1.$ Moreover,
$Y^{(1)} = \sum_{i=0}^{k-1} Y^{(1)}_i$ implies
$Y_{i(1)} = Y_{(1)}$ due to identity Jacobians of $Y$ with respect to $Y_i$
for $i=0,\ldots,k-1$ and hence 
Lemma~\ref{lem:ls2}. 
\end{proof}

\begin{theorem} \label{the}
Adjoints of systems $A \cdot \X = \B$ of $n$ linear equations with invertible 
$A \in \R^{n \times n}$ and right-hand side $\B \in \R^n$
are evaluated at the 
primal solution $\X \ass A^{-1} \cdot \B \in \R^n$
by Eqn.~(\ref{als}).
\end{theorem}
\begin{proof}
Differentiation of $A \cdot \X=\B$ with respect to $A$ and $\B$ yields
the tangent system
$$
A^{(1)} \cdot \X +
A \cdot \X^{(1)}=\B^{(1)}
$$
which implies
$$
\X^{(1)} \ass A^{-1} \cdot \B^{(1)} \cdot I_n- A^{-1} \cdot A^{(1)} \cdot \X 
$$
with identity $I_n \in \R^{n \times n}.$ 
Lemma~\ref{lem:ls2} yields 
	\begin{align*}
		\B_{(1)} &\ass A^{-T} \cdot \X_{(1)} \cdot I_n^T\\
		A_{(1)} &\ass -\underset{=\B_{(1)}}{\underbrace{A^{-T} \cdot \X_{(1)}}} \cdot \X^T 
	\end{align*}
and hence
Eqn.~(\ref{als}).
\end{proof}

\section{Conclusion}
\label{sec:conclusion}

As observed previously by various authors a possibly available factorization 
of $A$ can be reused both for the tangent 
($A \cdot \X^{(1)}=\B^{(1)}-A^{(1)} \cdot \X$) 
and the adjoint ($A^T \cdot \B_{(1)} = \X_{(1)}$) systems. The additional
worst case computational cost of $O(n^3)$ can thus be reduced to $O(n^2).$
Higher-order tangents [adjoints] of linear systems amount to repeated 
solutions of linear systems with the same [transposed] system matrix combined 
with tangent [adjoint] BLAS.

\end{document}